\documentclass[12pt]{amsart}
\usepackage{graphicx} 

\usepackage{biblatex} 
\usepackage{commath}
\usepackage{amsfonts}
\usepackage{amsmath}
\usepackage{setspace}

\usepackage{breqn}

\usepackage[margin=1in]{geometry}

\usepackage{graphicx}

\usepackage{mathrsfs}

\usepackage{chngcntr}
\counterwithin{equation}{section}
\usepackage[dvipsnames]{xcolor}

\newtheorem{theorem}{Theorem}[section]

\newtheorem{definition}{Definition}[section]
\newtheorem{lemma}{Lemma}[section]
\newtheorem{proposition}{Proposition}[section]
\newtheorem{corollary}{Corollary}[section]
\newtheorem{remark}{Remark}[section]

\usepackage{comment}

\newcommand{\Ind}{\mathrm{Ind}}
\newcommand{\Det}{\mathrm{Det}}
\newcommand{\Aut}{\mathrm{Aut}}
\newcommand{\sech}{\mathrm{sech}}
\newcommand{\End}{\mathrm{End}}
\newcommand{\id}{\mathrm{id}}
\newcommand{\An}{\mathrm{An}}
\newcommand{\supp}{\mathrm{supp\, }}

\begin{document}

\title{Analytic continuation of Toeplitz operators and commuting families of $C^*-$algebras}

\author{Khalid Bdarneh}

\address{Department of Mathematics, University of Jordan, Amman, Jordan}
\email{K.bdarneh@ju.edu.jo}
\author{Gestur \'Olafsson }
\footnote{The research of G. \'Olafsson was partially supported by Simons grant 586106.}
\address{ Department of Mathematics, Louisiana State University, Baton Rouge, LA 70803, USA}
\email{olafsson@lsu.edu}

\maketitle
\begin{abstract}
 We consider the Toeplitz operators on the weighted Bergman spaces over the unit ball $\mathbb{B}^n$ and their analytic continuation. We proved the commutativity of the $C^*-$algebras generated by the analytic continuation of Toeplitz operators with a special class of symbols that satisfy an invariant property, and we showed that these commutative $C^*-$algebras with symbols invariant under compact subgroups of $SU(n,1)$ are completely characterized in terms of restriction to multiplicity free representations. Moreover, we extended the restriction principal to the analytic continuation case for suitable maximal abelian subgroups of $SU(n,1)$, we obtained the generalized Segal-Bargmann transform and we showed that it acts as a convolution operator. Furthermore, we proved that Toeplitz operators are unitarly equivalent to a convolution operator and we provided integral formulas for their spectra.

\end{abstract}

\section{introduction}
Commutative $C^*-$algebras provide a framework for studying spectral properties of operators. The Gelfand-Naimark theorem states that every commutative $C^*-$algebra is $*-$ isomorphic to $C_0(X)$ for some locally compact Hausdorff space $X$.

One interesting problem is to find families of commutative $C^*-$algebras generated by Toeplitz operators. The algebra of Toeplitz operators is dense in the algebra of bounded operators under the strong operator topology. However, restricting the class of symbols ensures the existence of rich commutative families of Toeplitz operators. For example, consider the Bergman space on the unit disc $\mathbb{D}$, the $C^*-$algebra generated by Toeplitz operators with essentially bounded symbols is commutative if and only if there is a pencil of hyperbolic geodesics such that the symbols of the Toeplitz operators are constant on the cycles of this pencil \cite{grudsky2006commutative}. In other words, we need the symbols to be invariant under the maximal abelian subgroups of $\Aut(\mathbb{D})$. This result was extended for the unit ball $\mathbb{B}^n$ for the case $n>1$ in \cite{Quiroga-Barranco2007}. In particular, the $C^*-$algebra generated by Toeplitz operators with essentially bounded symbols that are invariant under maximal abelian subgroups of $\Aut(\mathbb{B}^n)$ is commutative. Moreover, the authors in \cite{Dawson2015} used tools from representation theory to extend this result to bounded symmetric domains.

 In this work we will construct commuting families of $C^*-$algebras generated by the analytic continuation of the Toeplitz operators, and we will develop the restriction principal for suitable maximal abelian subgroups of $SU(n,1)$ which will allow us to find the spectral decomposition of the analytic continuation of the Toeplitz operators.
 
Let $\mathbb{B}^n$ be the unit ball in $\mathbb{C}^n$, and consider the weighted Bergman space \[A_{\lambda}^2(\mathbb{B}^n)=\lbrace f\in L^2(\mathbb{B}^n,d\mu_{\lambda})\mid f \text{ is holomorphic on } \mathbb{B}^n \rbrace\]
with respect to the probability measure
\[d\mu_{\lambda}(z)=\dfrac{\Gamma(\lambda)}{\pi ^n\Gamma(\lambda-n)}(1 - \vert z\vert ^2)^{\lambda-n-1}dz, \quad \lambda>n \]
the Toeplitz operator is the operator
\begin{align*}
    T_{\varphi}^{(\lambda)} : & \,A_{\lambda}^2(\mathbb{B}^n)\rightarrow A_{\lambda}^2(\mathbb{B}^n)\\
    & T_{\varphi}^{(\lambda)}=P_{\lambda}(\varphi f)
\end{align*}
where $P_{\lambda} : L^2(\mathbb{B}^n,d{\mu_{\lambda}) \rightarrow A_{\lambda}^2(\mathbb{B}^n)}$ is the orthogonal projection \[ P_{\lambda}f(x)=\int_{\mathbb{B}^n} f(y) K_{\lambda}(x,y) d\mu_{\lambda}(y)\]

This definition of the Toeplitz operator works when $\lambda>n$. However, the weighted Bergman spaces $A_{\lambda}^2(\mathbb{B}^n)$ reduces to $\lbrace 0 \rbrace$ when we consider $\lambda\leq n$. The kernel of $A_{\lambda}^2(\mathbb{B}^n)$ is given by \[ K_{\lambda}(z,w)=(1-z.\bar{w})^{-\lambda}\] So, a natural question would be, for which values of $\lambda$ is $K_{\lambda}$ a reproducing kernel for a Hilbert space of holomorphic functions on $\mathbb{B}^n$? We will denote this set by $W$. A half line contained in $W$ was found by Vergne and Rossi \cite{vergne1976analytic}. However, using algebraic methods Wallach determined the entire set $W$ by showing it also contains a discrete set of points \cite{wallach1979analytic}.

The authors in \cite{Chailuek2010}  constructed Hilbert spaces of holomorphic functions with reproducing kernel
\[K_{\lambda}(z,w)=(1-z.\bar{w})^{-\lambda} \] 
with weights $ \lambda>0$, and they proved these spaces coincide with the weighted Bergman spaces when we restrict the values of $ \lambda$ to $ \lambda>n$. 

In \cite{10.1007/BF02384872} the authors generalized the work of Chailuek and Hall that appeared in \cite{Chailuek2010}, and proved the existence of the analytic continuation of weighted Bergman spaces to include all weights $ \lambda$ in $\mathbb{R}-V$ where $ V$ a discrete set. These spaces are Sobolev spaces of holomorphic functions of certain order. Moreover, they proved the existence and uniqueness of the analytic continuation of the Toeplitz operators for certain classes of symbols.

\section{Analytic continuation of Toeplitz operators}

The definition of weighted Bergman spaces were extended to include the weights $\lambda>0$, where these spaces are Sobolev spaces of holomorphic functions. The construction of theses spaces was introduced in \cite{Chailuek2010}, we will present a brief description of these spaces.

Consider the Euler operator \[N:=\sum_{j=1}^{n} z_j\dfrac{\partial}{\partial z_j}\] then $N z^r=|r|z^r$ for all multi-indices $r$. For $\lambda>0$, choose a non-negative integer $m$ such that $\lambda+2m>n$. Denote by $\mathcal{O}(\mathbb{B}^n)$ the space of holomorphic functions on the unit ball. Define the space \[ A_{\lambda,\#}^2(\mathbb{B}^n):= \lbrace f\in\mathcal{O}( \mathbb{B}^n) \mid N^k f\in A_{\lambda}^2(\mathbb{B}^n, d\mu_{\lambda+2m}) \text{ for every } 0\leq k\leq m\, \rbrace \]
This space with inner product defined by \[<f,g>_{\lambda}:=<Af,Bg>_{\lambda+2m}\] is a reproducing kernel Hilbert space, and the kernel is given by \[ K_{\lambda}(x,y)=(1 - <x,y>)^{-\lambda} \]
The monomials $\lbrace z^r \rbrace$ forms an orthogonal basis for the space  $A_{\lambda,\#}^2(\mathbb{B}^n)$, and direct calculations show that \[ <z^r,z^s>_{\lambda,\#}\,=\,\delta(r,s) \dfrac{r!\, \Gamma(\lambda)}{\Gamma(\lambda+|r|)} \]
where $\delta(r,s)=1$ when $r=s$, and zero otherwise.

\begin{definition}\cite{10.1007/BF02384872}
Assume for $f\in \mathcal{O}(\overline{\mathbb{B}^n})$ the function $F_{\lambda,f}(z):=T_{\varphi}^{(\lambda)}f(z)$ can be analytically continued from $\lambda>n$ to some neighborhood of $\lambda_{0}>0$, and the operator $T: f \mapsto F_{\lambda,f}$ extends to a bounded operator on $A_{\lambda_0,\#}^2(\mathbb{B}^n)$, then we say the Toeplitz operator $T_{\varphi}^{(\lambda_0)}$ exists and we set $T= T_{\varphi}^{(\lambda_0)}$.
\end{definition}

Let $Z:=\lbrace \lambda\in\mathbb{C}-\lbrace n, n-1, n-2, n-3,...\rbrace \mid C(\lambda)=0\rbrace$, where \[C(\lambda):=\int_{\mathbb{B}^n}(1-|z|^2)^{\lambda-n-1}dz\] For $k=0,1,2,...$ define the space \[ BC^k(\mathbb{B}^n):= \lbrace f\in C^k(\Omega): \partial^\nu \overline{\partial^\mu} f\in L^{\infty}(\Omega) \text{ for  every }  |\nu|+|\mu|\leq k \rbrace \] 
  
  The following theorem shows that the existence of the analytic continuation of Toeplitz operators is guaranteed when the symbol $\varphi$ is in $BC^k(\Omega)$.
 
 \begin{theorem}[\cite{10.1007/BF02384872}]
 Let $k\ge 0$ be an integer. If $\varphi\in BC^k(\mathbb{B}^n)$, then $T_{\varphi}^{(\lambda)}$ exists for all $\lambda\in (n-k,\infty)-(Z\cup\lbrace-2,-3,-4,...\rbrace)$.
 \end{theorem}

From now on, when we write $\lambda>0$, we mean $\lambda\in\mathbb{R}^{+}-\lbrace n,n-1,n-2,..., 0 \rbrace$.

\section{$C^{*}$-algebra Generated by Toeplitz Operators}
Representation theory can be used to construct commuting families of $C^*-$algebras generated by Toeplitz operators. For example, it was shown in \cite{Dawson2015} that certain multiplicity free representations corresponds to commuting families of Toeplitz operators with symbols admitting an invariant property. In particular, if $G$ is a Lie group of type I and $H\subset G$ is a compact subgroup, then the restriction of the holomorphic discrete series representation $\pi_{\lambda}|_H$ is multiplicity free if and only if the family of Toeplitz operators over $A_{\lambda}^2(\mathbb{B}^n)$ with $H$-invariant symbols is commutative \cite[Theorem 6.4]{Dawson2015}.
In this section we will construct commuting families of $C^*-$algebras, these families are generated by the analytic continuation of Toeplitz operators. In particular, we will extend \cite[Theorem 2]{Dawson2018} to include the analytic continuation case. Moreover, we will prove a density theorem for the $C^*-$algebra generated by the analytic continuation of Toeplitz operators. 

Let $G:=SU(n,1)$, and $H$ be a subgroup of $G$. 
  For $\lambda>0$, we consider the unitary representation of the group $SU(n,1)$ on $A_{\lambda}^2(\mathbb{B}^n)$, defined by \[ \pi_{\lambda}(g)f(z)=j_{\lambda}(g^{-1},z)f(g^{-1}.z)\] 
   where \[j_{\lambda}:SU(n,1)\times\mathbb{B}^n\rightarrow \mathbb{C}  \] \[ j_{\lambda} \left( \begin{pmatrix}
a & v \\ w^t & d 
\end{pmatrix},z \right)=(w^tz+d)^{-\lambda}\] The map $j_\lambda$ is lifted to $\widetilde{SU(n,1)}\times \mathbb{B}^n$ to make sure it is well defined for all $\lambda>0$.

\begin{proposition}
Let $\mathcal{L}(A_{\lambda}^2(\mathbb{B}^n))$ be the space of bounded linear operators on the weighted Bergman space $A_{\lambda}^2(\mathbb{B}^n)$. For every $\lambda>0$, the map 
\begin{align*}
    BC^n(\mathbb{B}^n) & \rightarrow \mathcal{L}(A_{\lambda}^2(\mathbb{B}^n))\\
    &\varphi \longmapsto T_{\varphi}^{(\lambda)}
\end{align*}
 is injective.
\end{proposition}
\begin{proof}
    Fix $\lambda>0$. Suppose $\varphi\in BC^n(\mathbb{B}^n) $ and $T_{\varphi}^{(\lambda)}=0$, then $< T_{\varphi}^{(\lambda)}z^n,z^m>_{\lambda}=0$ for every monomials $z^n$ and $z^m$. There exists $\varphi_\epsilon\in C^{\infty}(\mathbb{B}^n)$ such that $\norm{\varphi_\epsilon}_{\infty}\leq \norm{\varphi}_\infty$, and $\varphi_{\epsilon}\rightarrow \varphi$ pointwise as $\epsilon\searrow 0$ . Therefore $T_{\varphi_\epsilon}^{(\lambda)}\rightarrow T_{\varphi}^{(\lambda)}$ in the weak operator topology. Since $\varphi_{\epsilon}$ is holomorphic, then $T_{\varphi_\epsilon}^{(\lambda)}$ is the multiplication operator.
     So we have:
     \begin{align*}
         0=<T_{\varphi}^{(\lambda)}z^n,z^m>_{\lambda}=&\lim_{\epsilon\rightarrow 0}<T_{\varphi_\epsilon}^{(\lambda)}z^n,z^m>_{\lambda}\\ 
         &=\lim_{\epsilon\rightarrow 0}<\varphi_\epsilon z^n, z^m>_{\lambda}\\
         &= <\varphi z^n,z^m>_{\lambda}
     \end{align*}
     which means $\varphi$ is orthogonal to the monomials on $\mathbb{B}^n$. Therefore, $\varphi=0$
\end{proof}

Let $ \varphi\in BC^n(\mathbb{B}^n)$ and $h\in\widetilde{SU(n,1)}$, define $\varphi_{h}(z):=\varphi(h^{-1}z)$. Let $H$ be a subgroup of $\widetilde{SU(n,1)}$, we say $\varphi$ is $H$-invariant if $\varphi_h(z)=\varphi(z)$ for every $h\in H$, and for a.e $z\in \mathbb{B}^n$.

 The following proposition is the cornerstone in finding commuting families of Toeplitz operators. In particular, we need the Toeplitz operators to commute with the representation $\pi_{\lambda}$, and this can be done by restricting the representation $\pi_{\lambda}$ to suitable subgroup of $\widetilde{G}$.
 The proof of the following proposition follows from \cite{Dawson2015} in the case where $\lambda>n$, and the uniqueness of the analytic continuation for the case $\lambda>0$.

\begin{proposition}
\label{Toeplitz operator commute with representation}
 Let $\lambda>0$. For every Toeplitz operator $T_{\varphi}^{(\lambda)}$ we have \[  \pi_{\lambda}(g)\circ T_{\varphi}^{(\lambda)}=T_{\varphi_g}^{(\lambda)}\circ \pi_{\lambda}(g)\]
\end{proposition}

\begin{corollary}
\label{Generalized Toeplitz operator commute with representation}
Let $\widetilde{H}$ be a subgroup of $\widetilde{SU(n,1)}$. The symbol $\varphi$ is $\widetilde{H}$-invariant if and only if $T_{\varphi}^{( \lambda)}$ intertwine with $\pi_{\lambda}|_{\widetilde{H}}$ 
\end{corollary}

Multiplicity free representations corresponds to commuting families of Toeplitz operators, we will show this also remains true when we consider the analytic continuation case. Let $H$ be a Lie group of type I, and  $\pi:H\rightarrow \mathcal{L(H)}$ be a unitary representation of $H$ on a Hilbert space  $\mathcal{H}$. The representation $\pi$ admits a unique direct integral decomposition (up to unitary equivalence) \[ \pi\simeq \int_{\widehat{H}}^{\bigoplus}m_{\pi}(\tau)\tau d\mu(\tau) \] where $\widehat{H}$ is the set of equivalence classes of irreducible unitary representation of $H$, $\mu$ is a Borel measure on $\widehat{H}$, and $m_{\pi}:\widehat{H}\rightarrow \mathbb{N}\cup \lbrace\infty\rbrace $ is the multiplicity function \cite{Kobayashi2005}. A unitary representation $\pi$ is called multiplicity free if $m_{\pi}(\tau)\leq 1$ for almost all $\tau\in\widehat{H}$.

Let $\End_H(\mathcal{H})$ be the set of operators that intertwine with the representation $\pi$, the ring $\End_H(\mathcal{H})$ is commutative if and only if the representation $\pi$ is multiplicity free \cite{Kobayashi2005}. Let ${H}$ be a subgroup of ${SU(n,1)}$, and $BC^n(\mathbb{B}^n)^{{H}}$ be the set of symbols in $BC^n(\mathbb{B})^n$ that are ${H}$-invariant. We will use $\mathcal{T}^{( \lambda)}(BC^n(\mathbb{B}^n)^H)$ to denote the $C^*$-algebra generated by Toeplitz operators with $H-$ invariant symbols on the Hilbert space $A_{\lambda}^2(\mathbb{B}^n)$.

\begin{lemma}[\cite{10.1007/BF02384872}]
Let $\Omega$ be a domain in $\mathbb{C}^n$ and $F(x,y)$ be defined on $\Omega\times \bar{\Omega}$ such that $F$ is holomorphic in first variable and anti-holomorphic in the second variable. If $F(x,\bar{x})=0$ for every $x$ in $\Omega$, then $F$ is identically zero on $\Omega\times \bar{\Omega}$.
\end{lemma}

The following lemma appeared in the proof of Theorem 2 in \cite{10.1007/BF02384872}. We will need it to show that the space of  Toeplitz operators with $H-$invariant symbols is dense in $\End_{\widetilde{H}}(A_{\lambda}^2(\mathbb{B}^n))$.
\begin{lemma}
\label{zero on the diagonal}
Let $q$ and $r$ be integers, and  $W$ be a finite dimensional subspace of $A_{\lambda}^2(\mathbb{B}^n)$ of dimension $qr$. Assume $W$ has basis $\lbrace f_1,...,f_q,g_1,...,g_r \rbrace$. If $u\in \mathbb{C}^{q\times r}$ is such that \[ \sum_{j}\sum_{i} <T_{\varphi}^{(\lambda)}f_i,g_j>u_{ij}=0 \] for every $\varphi\in BC^n(\mathbb{B}^n)$, then $u=0$.
\end{lemma}

\begin{proof}  
Suppose we have $u\in\mathbb{C}^{q\times r}$ such that \[ \sum_{j}\sum_{i} <T_{\varphi}^{(\lambda)}f_i,g_j>u_{ij}=0 \] for every $\varphi$. Then \[ \int_{\Omega} \varphi(z) \sum_j\sum_i u_{ij}f_i(z)\overline{g_j(z)}.d\mu(z)=0 \] for all $\varphi\in BC^n(\mathbb{B}^n)$. Therefore, we have
\begin{equation}
  \sum_j\sum_i u_{ij}f_i(z)\overline{g_j(z)}=0  
\end{equation} almost everywhere on $\Omega$. Since the left hand side of equation 3.1 is continuous, then the equation holds everywhere on $\Omega$. Define the function \[ F(x,y)=\sum_j\sum_i u_{ij}f_i(x)\overline{g_j(\overline{y})} \] then $F(z,\bar{z})=0$ on $\Omega$. Therefore, the function $F$ is zero on $\Omega\times \overline{\Omega}$. Since the functions $f_i's$ and $g_i's$ are linearly independent, then $u_{ij}=0$ for every $i,j$. 
\end{proof}

 For $\lambda>n$, the $C^*-$algebra $\mathcal{T}^{(\lambda)}(BC^n(\mathbb{B}^n)^{{H}})$ generated by Toeplitz operators with $H$- invariant symbols, is dense in $\End_{\widetilde{H}}(A_{\lambda}^2(\mathbb{B}^n))$ under the strong operator topology, the following proposition shows that this density result still holds for the analytic continuation of Toeplitz operators with $H$- invariant symbols, the proof is close to the proof of \cite[Theorem 2]{10.1007/BF02384872}. If $F(\lambda,x)$ is a function that depends on $\lambda$, then we will use the notation $\An_{\lambda_0}(F)$ to denote the analytic continuation of the function $F$ to a neighborhood of $\lambda_0$.

\begin{proposition}\label{generalized Englis density theorem}
 Let $H$ be a compact subgroup of $SU(n,1)$. For every $\lambda>0$, the space $\mathcal{T}^{(\lambda)}(BC^n(\mathbb{B}^n)^{H})$ is dense in $\End_{\widetilde{H}}(A_{\lambda}^2(\mathbb{B}^n))$ the space of $\widetilde{H}$-intertwining operators on $A_{\lambda}^2(\mathbb{B}^n)$ under the strong operator topology.
 \end{proposition}
 \begin{proof}
   Since $H$ is compact, then $A_{\lambda}^2(\mathbb{B}^n)$ decomposes into an orthogonal finite dimensional $\widetilde{H}-$invariant subspaces.  Let $\lambda>0$ and $W$ be a $\widetilde{H}-$invariant finite dimensional subspace of $A_{\lambda_0}^2(\mathbb{B}^n)$ with basis $f_1,...,f_q,g_1,...,g_r$. Let $T$ be a bounded operator on $A_{\lambda_0}^2(\mathbb{B}^n)$, we will show that there exists $\varphi\in BC^n(\mathbb{B}^n)$ such that \[ <Tf_i,g_j>\,\, =\,\, <T_{\varphi}^{(\lambda_0)}f_i,g_j>  \]
 
 Define the operator \[ R:BC^n(\mathbb{B}^n)\rightarrow \mathbb{C}^{q\times r} \] \[ (R\varphi)_{ij}:=\,\, <T_{\varphi}^{(\lambda_0)}f_i,g_j> \] which can be written as  \[ (R\varphi)_{ij} =\An_{\lambda_0}(<T_{\varphi}^{(\lambda)}f_i,g_j>) \]
 Suppose $u\in\mathbb{C}^{q\times r}$ is orthogonal to the image of the operator $R$, that is   \[ \sum_{j}\sum_{i} <T_{\varphi}^{(\lambda_0)}f_i,g_j>\overline{u_{ij}}=0 \] 
  for every $\varphi\in BC^n(\mathbb{B}^n)$. So, we have  \[\An_{\lambda_0} ( \sum_{j}\sum_{i} <T_{\varphi}^{(\lambda)}f_i,g_j>\overline{u_{ij}})=0 \]
 which implies \[ \sum_{j}\sum_{i} <T_{\varphi}^{(\lambda)}f_i,g_j>\overline{u_{ij}}=0 \]
 Therefore, by Lemma \ref{zero on the diagonal} we have $u_{ij}=0$. Which means that the image of $R$ is $\mathbb{C}^{q\times r}$.

 Since $W$ was an arbitrary finite dimensional $\widetilde{H}-$invariant subspace, then $\mathcal{T}^{(\lambda)}(BC^n(\mathbb{B}^n)^{{H}})$ is dense in $\End_{\widetilde{H}}(A_{\lambda}^2(\mathbb{B}^n))$ in the weak operator topology. Since $\mathcal{T}^{(\lambda)}(BC^n(\mathbb{B}^n)^H)$ is a convex subspace of $\End_{\widetilde{H}}(A_{\lambda}^2(\mathbb{B}^n))$, then $\mathcal{T}^{(\lambda)}(BC^n(\mathbb{B}^n)^H)$ is dense in $\End_{\widetilde{H}}(A_{\lambda}^2(\mathbb{B}^n))$ in the strong operator topology.
  \end{proof}

One of our main goals is to establish a connection between the commutativity of a subspace of analytic continuation of Toeplitz operators and multiplicity free representations. First, we need the following Lemma. The proof can be found in \cite[Lemma 6.3]{Dawson2015}.

\begin{lemma}
    \label{commutitivity of dense subspace}
    Suppose that $V$ is a linear subspace of the space $\mathcal{L}(\mathcal{H})$ of bounded linear operators on a Hilbert space $\mathcal{H}$. If $V$ consists of operators that commutes, then the closure of $V$ in the strong operator topology also consists of operators that commutes.
\end{lemma}

\begin{theorem}
\label{multiplicity free and commutativity}
Let $\lambda >0$.
\begin{itemize}
\item[\rm 1)] If $\End_{\widetilde{H}}(A_{\lambda}^2(\mathbb{B}^n))$ is commutative, then $\mathcal{T}^{(\lambda)}(BC^n(\mathbb{B}^n)^H)$ is a commutative $C^*$-algebra.
\item[\rm 2)] If $\widetilde{H}$ is compact, then $\mathcal{T}^{(\lambda)}(BC^n(\mathbb{B}^n)^H)$ is  commutative if and only
 if $\pi_{\lambda}|_{\widetilde{H}}$ is multiplicity free.
 \end{itemize}
\end{theorem}
\begin{proof}
    
(1) Using proposition \ref{Toeplitz operator commute with representation}  we have every Toeplitz operator with $H-$invariant symbol commute with the representation $\pi$. Therefore $\mathcal{T}^{(\lambda)}(BC^n(\mathbb{B}^n)^H)\subset \End_{\widetilde{H}}(A_{\lambda}^2(\mathbb{B}^n))$, and the commutativity of $\End_{\widetilde{H}}(A_{\lambda}^2(\mathbb{B}^n))$ implies  $\mathcal{T}^{(\lambda)}(BC^n(\mathbb{B}^n)^H)$ is also commutative.

(2) If the $C^*-$algebra $\mathcal{T}^{(\lambda)}(BC^n(\mathbb{B}^n)^H)$ is commutative, then by proposition \ref{generalized Englis density theorem} and lemma \ref{commutitivity of dense subspace}, we have the space $\End_{\widetilde{H}}(A_{\lambda}^2(\mathbb{B}^n))$ is commutative. Therefore, the representation $\pi_{\lambda}|_{\widetilde{H}}$ is multiplicity free. Conversely, if the representation $\pi_{\lambda}|_{\widetilde{H}}$ is multiplicity free, then $\End_{\widetilde{H}}(A_{\lambda}^2(\mathbb{B}^n))$ is commutative, and the commutativity of  $\mathcal{T}^{(\lambda)}(BC^n(\mathbb{B}^n)^H)$ follows from (1).
\end{proof}
 
\section{Spectral Representation of Toeplitz Operators}
On the unit ball $\mathbb{B}^n$, the Segal-Bargman transform was used to derive a formula for the spectrum of the Toeplitz operators, which was given as an integral formula in terms of the symbols \cite{Quiroga-Barranco2007,Quiroga-Barranco2008}. We will extend the results in \cite{Dawson2018} to the case when the weight $\lambda>0$. In particular, we will consider the class of symbols that are invariant under maximal abelian subgroups of $SU(n,1)$, and  we will show that Toeplitz operators can be realized as a convolution operator. 
  
\subsection{ The Restriction Principle}
 Let $M\subset \mathbb{B}^n$ be a submanifold. If $f\mapsto f|_M$ is injective on the space of holomorphic functions, then $M$ is called \textbf{restriction injective} \cite{Dawson2018}.
  Let $H$ be a closed subgroup of $G:=SU(n,1)$, and $\widetilde{H}$ be the inverse image of $H$ in the universal covering of $G$. Assume that the submanifold $M:=H.\,z_0=\widetilde{H}.\,z_0$ is restriction injective, then there exists a measure on $M$ that is invariant under the submanifolds $H$ and $\widetilde{H}$ \cite{Dawson2018}.

  \begin{proposition}
\cite{Dawson2018}  
   For $\lambda>0$, the map $\chi_{\lambda}: {\widetilde {H}_{z_0}}\rightarrow\mathbb{C} $, defined by $\chi_{\lambda}(h)=  j_{\lambda}(h,z_0)^{-1} $ is a unitary character satisfying
    \[ j_{\lambda}(hk,z_0)=j_{\lambda}(h,z_0)\chi_{\lambda}(k)^{-1}  \]
 for all $h\in \widetilde{H},\,k\in{\widetilde{H}_{z_0}}$.
  \end{proposition}

 As $\chi_{\lambda}$ is a character of $\widetilde{H}_{z_0}$, we can consider the induced representation  $\tau_\lambda:=\Ind_{\widetilde{H}_{z_0}}^{\widetilde{H}}\chi_\lambda$. The Hilbert space for the induced representation is denoted by $L_{\chi_{\lambda}}^2(M,d\mu)$ which consists of the functions $f: \tilde{H}\rightarrow \mathbb{C}$ that are square integrable with respect to the invariant measure $\mu$ and satisfy  
  \begin{equation}
      f(kh)=\chi_{\lambda}(h)^{-1}f(k)
  \end{equation}
  for all $k\in\widetilde{H}$ and $ h\in\widetilde{H}_{z_0}$. Moreover, $\widetilde{H}$ acts unitarily on  $L_{\chi_{\lambda}}^2(M,d\mu)$ by \[h.f(k)=f(h^{-1}k)\]  Notice that the space $M$ is identified with $\widetilde{H}/\widetilde{H}_{z_0}\cong H/H_{z_0}$.

  There are $n+2$ conjugacy classes of maximal abelian subgroups of $\Aut(\mathbb{B}^n)$ \cite{Raul-Moment-maps}, we will list their representatives in the next section. In this section, when we write a subgroup $H$ of $SU(n,1)$ we mean a maximal abelian subgroup that is either quasi-elliptic, quasi-parabolic or quasi-hyperbolic. We will associate a differential operator $N$ for each one of them.

  For the quasi-elliptic abelian subgroup $E(n)$, we define the operator $N_{E(n)}:=0$ to be the zero operator. For the Quasi-parabolic subgroup$P(n)$, we define the operator  $N_{P(n)}:=y\dfrac{\partial}{\partial\,y}$. Finally, for the quasi-hyperbolic subgroup $H(n)$, we define the operator \[N_{H(n)}:=\dfrac{\sinh\,s}{\cosh\,s}\dfrac{\partial}{\partial\,s}\]

Suppose $n\geq\lambda>0$, let $m$ be the smallest non-negative integer such that $\lambda+2m>n$. Let $H$ ba a maximal abelian subgroup of $SU(n,1)$, define the operators:

\begin{equation}
\label{Operator A and B}
\begin{aligned}
    A_{\lambda} &:=(I+\dfrac{N_H}{\lambda +2m-1})...(I+\dfrac{N_H}{\lambda+m}) \\
    B_{\lambda} &:=(I+\dfrac{N_H}{\lambda+m-1})...(I+\dfrac{N_H}{\lambda})\\
\end{aligned}
\end{equation}
the operators $A_{\lambda} $ and $B_{\lambda}$ are injective. Define the space $\mathcal{V}_{\lambda}$ to be the space of analytic functions $f:\widetilde{H}\rightarrow \mathbb{C}$ such that $A_{\lambda}B_{\lambda}\,f\in L_{\chi_{\lambda+2m}}^2(M,d\mu)$ and satisfy \[ f(hh_1)=\chi_{\lambda}(h_1)^{-1}f(h) \] for all $h\in \Tilde{H}$, and $h_1\in \Tilde{H}_{z_0}$. Moreover, the space $\mathcal{V}_{\lambda}$ is equipped with the inner product \[<f,g>_{\lambda}:=<A_{\lambda}B_{\lambda} \, f,A_{\lambda}B_{\lambda}\, g>_{L_{\chi_{\lambda+2m}}^2}\] 
If $\lambda>n$, then we take $A_{\lambda}=B_{\lambda}=I$.

Define the map $D_{\lambda}: \widetilde{H}\rightarrow \mathbb{C}$ by $D_{\lambda}(h)=j_{\lambda}(h,z_0)$, the cocycle relation of $j_\lambda$ implies \[D_{\lambda}(hh_1)=\chi_{\lambda}(h_1)^{-1}D_{\lambda}(h)\]
If we apply the operators $A_\lambda B_\lambda$ on $D_{\lambda}$, then we have 
\begin{equation}
\label{AB applied to D}
   A_{\lambda}B_{\lambda}\, D_{\lambda}(h)=c D_{\lambda+2m}(h)
\end{equation}
in the quasi-elliptic and quasi-parabolic case, where $c$ is a constant that depends on $h$. Therefore, $D_{\lambda}$ belongs to the space $\mathcal{V}_{\lambda}$. In fact, equation \ref{AB applied to D} does not hold for the quasi-hyperbolic case, but we still have $D_{\lambda} \in \mathcal{V}_{\lambda}$ even in this case.

Now, define the restriction operator $R_{\lambda}: A_{\lambda}^2(\mathbb{B}^n)\rightarrow \mathcal{V}_{\lambda}$ by 
 \[R_{\lambda}(f)(h)=D_{\lambda}(h)f|_M(h.z_0) \]
and let $\mathcal{S}_{\lambda}$ be the closure of $R_{\lambda}(A_{\lambda}^2(\mathbb{B}^n))$ in $\mathcal{V}_{\lambda}$. The point evaluation maps are continuous in $A_{\lambda}^2(\mathbb{B}^n)$, it follows that $R_{\lambda}$ is closed. Therefore, the operator $R_{\lambda}^*: \mathcal{S}_{\lambda}\rightarrow A_{\lambda}^2(\mathbb{B}^n)$ is well-defined and 
\begin{align*}
    R_{\lambda}^*f(z)=&< R_{\lambda}^*f,K_z>_\lambda = <f, R_{\lambda}K_z>_{\mathcal{S}_{\lambda}} \\
    =& \int_{\widetilde{H}/\widetilde{H}_{z_0}} A_{\lambda}B_{\lambda} f(h)\,\, \overline{A_{\lambda}B_{\lambda}\Big( D_{\lambda}(h) K(h.z_0,z)\Big)} \, d\mu(h)
\end{align*}
Using integration by parts, the operator $R_{\lambda}^*$ can be written as
\[ R_{\lambda}^* f(z)=\int_{\widetilde{H}/\widetilde{H}_{z_0}}  f(h)\,\, Q\,\overline{A_{\lambda}B_{\lambda}\Big( D_{\lambda}(h) K(h.z_0,z)\Big)} \, d\mu(h) \]
where $Q$ is a polynomial differential operator acting on $h$.

The following proposition shows that the restriction operator $R_{\lambda}$ intertwine the representations $\big ( \pi_{\lambda}|_{\widetilde{H}},A_{\lambda}^2(\mathbb{B}^n)\big )$ and $\big ( \tau_\lambda,\mathcal{S}_\lambda\big )$ where $\tau_\lambda$ is the left regular representation on $\mathcal{S}_{\lambda}$.

\begin{proposition}
    If $f\in A_{\lambda}^2(\mathbb{B}^n)$, and $h,h_1\in \widetilde{H}$, then $R_{\lambda}(\pi_{\lambda}(h_1)f)(h)=(\tau_{\lambda}(h_1)R_{\lambda}f)(h)$.
\end{proposition}
\begin{proof}
    Let $f\in A_{\lambda}^2(\mathbb{B}^n)$, $h$ and $h_1\in \widetilde{H}$, then we have
   \begin{align*}
     R_{\lambda}(\pi_{\lambda}(h_1)f)(h)=&D_{\lambda}(h)(\pi_{\lambda}(h_1)f)|_{M}(h.z_0) \\
     =& D_{\lambda}(h) j_{\lambda}(h_1^{-1},h.z_0)f(h_1^{-1}h.z_0)\\
     =& j_{\lambda}(h,z_0)j_{\lambda}(h_1^{-1},h.z_0)f(h_1^{-1}h.z_0)\\
     =& j_{\lambda}(h_1^{-1}h,z_0)f(h_1^{-1}h.z_0)\\
     =& \tau_{\lambda}(h_1)\big{(}j_{\lambda}(h,z_0)f(h.z_0)\big{)}\\
     =& (\tau_{\lambda}(h_1)R_{\lambda}f)(h)\qedhere
  \end{align*}
  \end{proof}  

We will need the following lemma in the next section. Following \cite{Dawson2018}, for $h,k\in\widetilde{H}$ we define the operator \[ R_{\lambda}(h,k)=D_{\lambda}(h)\, Q\overline{A_{\lambda}B_{\lambda} \big ( D_{\lambda}(k) K_{\lambda}(k.z_0,h.z_0)\big )} \]

  \begin{lemma}  \label{RR^* integral formula}
  Let $f\in\mathcal{S}_{\lambda}$, then \[ R_{\lambda}R_{\lambda}^*(f)(h)=\int_M f(k) R_{\lambda}(h,k)\,d\mu(k) \]
  \end{lemma}
  \begin{proof}
      Let $f\in\mathcal{S}_{\lambda}$, we have
  \[
   R_{\lambda}^*(f)(z)= \int_{\widetilde{H}/\widetilde{H}_{z_0}}  f(h)\, Q\,\overline{A_{\lambda}B_{\lambda}\Big( D_{\lambda}(h) K(h.z_0,z)\Big)} \, d\mu(h) \]
Therefore, $R_{\lambda}R_{\lambda}^*(f)$ becomes 
  
  \begin{align*}
      R_{\lambda}R_{\lambda}^*(f)(h)=&R_{\lambda}\big ( \int_{\widetilde{H}/{\widetilde{H}_z{_0}}} f(k) \, Q \, \overline{A_{\lambda}B_{\lambda} \big ( D_{\lambda}(k)K_{\lambda}(k.z_0,z) \big ) }.d\mu(k)\big ) (h) \\
      =&D_{\lambda}(h) \int_{\widetilde{H}/{\widetilde{H}_z{_0}}} f(k) \, Q\,\overline{A_{\lambda}B_{\lambda} \big ( D_{\lambda}(k)K_{\lambda}(k.z_0,h.z_0) \big ) }.d\mu(k)\\
      =& \int_M f(k) R_{\lambda}(h,k)\,d\mu(k)\qedhere
  \end{align*}
  \end{proof}
 
The operator $R_{\lambda}R_{\lambda}^{*}$ is a positive and closed, hence the square root $\sqrt{R_{\lambda}R_{\lambda}^{*}}$ is well defined and we the have polar decomposition \[R_{\lambda}^{*}=U_{\lambda}\sqrt{R_{\lambda}R_{\lambda}^{*}}\] For a unique unitary map $U_{\lambda}: \mathcal{S}_{\lambda}\rightarrow A_{\lambda}^2(\mathbb{B}^n)$. 
  The operator $U_{\lambda}$ is called the Segal-Bargmann transform see \cite{olafsson1996generalizations}. Since the operator $R_{\lambda}$ intertwine the representations $\pi_{\lambda}|_{\widetilde{H}}$, and $\tau_{\lambda}$, then $U_{\lambda}$ is a unitary $\widetilde{H}$-isomorphism. The above discussion can be summarized in the following theorem, which shows that \cite[Theorem 1]{Dawson2018} holds for $\lambda>0$.
  \begin{theorem}
  The Segal-Bargmann transform $U_{\lambda}:\Big{(}\mathcal{S}_{\lambda}, \tau_{\lambda} \Big{)} \rightarrow \Big{(}A_{\lambda}^2(\mathbb{B}^n),\, \pi_{\lambda}|_{\widetilde{H}} \Big{)}$ is a unitary $\widetilde{H}-$isomorphism.
  \end{theorem}
  
  If $\widetilde{H}$ is Type I subgroup, then the representation $\pi_{\lambda}|_{\widetilde{H}}$ admits a unique direct integral decomposition ( up to a unitary equivalence)
  \[ \pi_{\lambda}|_{\widetilde{H}}\simeq \int_{\widehat{\widetilde{H}}}^{\bigoplus}m_{\pi_{\lambda}}(\sigma)\sigma d\mu_{\lambda}(\sigma) \] where $\widehat{\widetilde{H}}$ is the set of equivalence classes of irreducible unitary representation of $\widetilde{H}$, $\mu_{\lambda}$ is a Borel measure on $\widehat{\widetilde{H}}$, and $m_{\pi_{\lambda}}:\widehat{\widetilde{H}}\rightarrow \mathbb{N}\cup \lbrace\infty\rbrace $ is the multiplicity function. Moreover, multiplicity free representations corresponds to commutative $C^*-$algebras, in particular,  Theorem \ref{multiplicity free and commutativity} says that If $H$ is compact, then $\mathcal{T}^{(\lambda)}(BC^n(\mathbb{B}^n)^H)$ is commutative if and only if $\pi_{\lambda}|_{\widetilde{H}}$ is multiplicity free. Additionally, if an operator commutes with the representation $\pi_{\lambda}$, then the operator also has a direct integral decomposition. Furthermore, Corollary \ref{Generalized Toeplitz operator commute with representation} says that Toeplitz operator $T_{\varphi}^{(\lambda)}$ is an intertwining operator for $\pi_{\lambda}$ if and only if the symbol $\varphi$ is $\widetilde{H}-$invariant. Since the Segal-Bargmann transform $U_{\lambda}$ commute with $\pi_{\lambda}$ then the diagonalization of the Toeplitz operator is 
  \[ U_{\lambda}^{*}T_{\varphi}^{(\lambda)}U_{\lambda}= \int_{\widehat{\widetilde{H}}}^{\bigoplus}m_{\varphi,{\lambda}}(\sigma)\id_{\mathcal{H}_{\sigma}} d\mu_{\lambda}(\sigma)\] where $m_{\varphi,{\lambda}}:\widehat{\widetilde{H}}\rightarrow \mathbb{C}$. The set $(m_{\varphi,{\lambda}}(\sigma))_{\sigma}$ is the \textit{spectrum} of the Toeplitz operator $T_{\varphi}^{(\lambda)}$.


\section{Maximal Abelian Subgroups of $SU(n,1)$}
Our main goal in this section is to describe the spectral decomposition of Toeplitz operators with symbols that are invariant under the maximal abelian subgroups of $SU(n,1)$. In particular, we will apply the restriction principle to write the Toeplitz operators as a convolution operator on certain sections of line bundles. In fact, we will show that the work in \cite{Dawson2018} can be generalized to the analytic continuation case.

The group $SU(n,1)$ realizes the group of all automorphisms of the unit ball $\mathbb{B}^n$. As mentioned before, there are n+2 conjugacy classes of maximal abelian subgroups of $\Aut(\mathbb{B}^n)$. It is easier to describe their action on an unbounded realization of the unit ball $\mathbb{B}^n$. Let $D_n$ be the unbounded domain defined by 
\[ D_n= \lbrace (z^{'},z_n)\in \mathbb{C}^{n-1}\times \mathbb{C} \mid Im(z_n)-|z^{'}|^2 >0 \rbrace   \]
The Cayley transform defines a biholomorphism $\mathbb{B}^n\rightarrow D_n$, which is given by \[ z\mapsto \dfrac{i}{1+z_n}(z^{'},1-z_n ) \] with inverse 
\[  z\mapsto \dfrac{1}{1-i z_n}(-2iz^{'}, 1+iz_n) \]
Moreover, if we define 
\[ C:= 
\begin{pmatrix}
\begin{smallmatrix}
  i \\
  && \ddots\\
  &&& i\\
  &&&& -i && i\\
  &&&&  1 && 1
  \end{smallmatrix}
  \end{pmatrix}
   \]
then the group $CSU(n,1)C^{-1}$ realizes the group of biholomorphisms of $D_n$.
Each maximal abelian subgroup is conjugate to one of the following representative \cite{Raul-Moment-maps,Dawson2018} :

\textbf{ Quasi-elliptic}: This maximal abelian subgroup corresponds to the maximal compact torus in $SU(n,1)$, it is given by the $\mathbb{T}^n-$action on the unit ball \[ t.z=(t_1z_1,...,t_nz_n) \] where $t\in \mathbb{T}^n$ and $z\in \mathbb{B}^n$.

\textbf{Quasi-parabolic}: This group is isomorphic to $\mathbb{T}^{n-1}\times \mathbb{R}$, and it acts on $D_n$ by \[ (t,y).(z^{'},z_n)=(tz^{'},z_n+y) \]  

\textbf{Quasi-hyperbolic}: This group is isomorphic to $\mathbb{T}^{n-1}\times \mathbb{R}^{+}$ and acts on $D_n$ by \[ (t,r).(z^{'}.z_n)=(rtz^{'},r^2 z_n) \]

\textbf{Nilpotent}: This group is isomorphic to $\mathbb{R}^{n-1}\times \mathbb{R}$ and acts on $D_n$ by \[ (b,s).(z^{'},z_n)=(z^{'}+b,z_n+2i<z^{'},b>+s+i|b|^2) \]

\textbf{Quasi-nilpotent}: This group is isomorphic to $\mathbb{T}^k \times \mathbb{R}^{n-k-1}\times \mathbb{R}^{+}$ where $1\leq k\leq n-2$, this group acts on $D_n$ by \[ (t,b,s).(z^{'},z^{''},z_n)=(tz^{'}, z^{''}+b,z_n+2i<z^{''},b>+s+i|b|^2) \]
We will restrict our work to the first three classes of subgroups. In particular, when we say a maximal abelian subgroup, we mean it is either Quasi-elliptic, Quasi-parabolic, or Quasi-hyperbolic. 

Let $H$ be a maximal abelian subgroup of $SU(n,1)$, and $\widetilde{H}$ be the subgroup of $\widetilde{SU(n,1)}$ that covers $H$. The character $\chi_{\lambda}: H_{z_0}\rightarrow \mathbb{C}$ can be lifted to $\widetilde{H}$. This can be done by constructing a holomorphic embedding $\widetilde{H}\hookrightarrow H/H_{z_0}\times \mathbb{R}$, and $\chi_{\lambda}$ extends to $\widetilde{H}$ by defining \[ \chi_{\lambda}\big{(} (h,x) \big{)}:=e^{2\pi i\lambda x} \] for all $h\in H/H_{z_0}$ and $x\in \mathbb{R}$. This construction will be clear when we examine each case of the maximal abelian subgroups of $SU(n,1)$ in the following subsections.

Let $f$ be in $\mathcal{S}_{\lambda}$, and define $\widetilde{f}:=f\chi_{\lambda}$ , we will show that the operator $R_{\lambda}R_{\lambda}^*$ can be written as a convolution operator. In particular, we will show that  \[R_{\lambda}R_{\lambda}^* f=\chi_{-\lambda}. (\widetilde{f}* \phi_H )\] where $\phi_H\in L^1(H/H_{z_0})$.
Therefore, the space $\mathcal{S}_{\lambda}$ can be realized as 
\[ \mathcal{S}_{\lambda}= \lbrace f \in \mathcal{V}_{\lambda}: 
\big{(} \text{For } \psi \in \widehat{H/H_{z_0}} \text{ such  that } \widehat{\phi_H}(\psi)=0\big{)} \, \mathcal{F}_{H/H_{z_0}} \widetilde{f}(\psi)=0 \rbrace \] 
where $\mathcal{F}_{H/H_{z_0}}$ is the Fourier transform on $H/H_{z_0}$.

 \begin{subsection}{Quasi-elliptic}
  Following \cite{Dawson2018}, the quasi-elliptic abelian subgroup corresponds to the maximal compact torus in $SU(n,1)$, and it is given by 
  \[E(n)= \begin{Bmatrix}
  
  k_{t,a}=\begin{pmatrix}
  
  \begin{smallmatrix}
  at_1 &\\
  & at_2\\
  & & \ddots\\
  &&& at_n\\
  &&&& a
  \end{smallmatrix}
  
  \end{pmatrix}
  
  \Bigg|\begin{smallmatrix}
   a, t_1,...,t_n\in\mathbb{T}\\
  \Det(k_{t,a})=1
  \end{smallmatrix}
  \end{Bmatrix}
  \]
  
  The subgroup in the universal covering of $SU(n,1)$ that corresponds to $E(n)$ can be identified with \[ \widetilde{E(n)}=\lbrace (t_1,t_2,...,t_n,x)\in \mathbb{T}^n\times\mathbb{R}\mid e^{2{\pi}i(n+1)x}t_1t_2...t_n=1\rbrace \] and the product on $\widetilde{E(n)}$ is then given by \[ (t,x_1).(s,x_2)=(ts,x_1+x_2)  \]
  Fix $z_0=\big( (2n)^{-1/2},...,(2n)^{-1/2}\big)\in\mathbb{B}^n$. The action of $E(n)$ on $z_0$ is given by \[k_{t,a}.z_0=\dfrac{1}{\sqrt{2n}}(t_1,...,t_n)  \] and the stabilizer of $z_0$ are given by
  \[ E(n)_{z_0}=\lbrace k_{t,a}\mid t_i's=1\rbrace =\lbrace k_{t,a} \mid  a^{n+1}=1\rbrace\simeq \mathbb{Z}_n\]
  \[ \widetilde{E(n)_{z_0}}=\lbrace (1,1,...,1,\dfrac{m}{n+1})\mid m\in\mathbb{Z} \rbrace\simeq \mathbb{Z} \]
  Let $q=(t,x)\in \widetilde{E(n)}$, then $D_{\lambda}(q)=e^{-2\pi i \lambda x}$. The  differential operators $A_\lambda$ and $B_\lambda$ (\ref{Operator A and B}) acts on $D_{\lambda}$ as an identity map. In particular, $A_\lambda B_\lambda(D_{\lambda}(q))=D_{\lambda}(q)$. Therefore, $D_{\lambda}\in \mathcal{V}_{\lambda}$ for every $\lambda>0$.

Let $f$ be in $\mathcal{S}_{\lambda}$, by Lemma \ref{RR^* integral formula} we have that the operator $R_{\lambda}R_{\lambda}^*$ is given by the integral 
\[  R_{\lambda}R_{\lambda}^*(f)(h)=\int_M A_{\lambda}B_{\lambda}(f(k)) R_{\lambda}(h,k)\,d\mu(k) \] 
where
\[R_{\lambda}(h,k)=D_{\lambda}(h)\overline{A_{\lambda}B_{\lambda} \big ( D_{\lambda}(k) K_{\lambda}(k.z_0,h.z_0)\big )} \]
Let $h=(t,x)$ and $k=(s,y)$. Since the differential operators $A_\lambda$ and $B_\lambda$ acts as the identity map, then the operator $R_{\lambda}(h,k)$ can be written as 
\begin{align*}
    R_{\lambda}(h,k)& =  D_{\lambda}(h)\overline{A_{\lambda}B_{\lambda} \big ( D_{\lambda}(k) K_{\lambda}(k.z_0,h.z_0)\big )} \\
   & =  D_{\lambda}(h)\overline{  D_{\lambda}(k) K_{\lambda}(k.z_0,h.z_0)}\\
    &=  D_{\lambda}(h)\overline{  D_{\lambda}(k)} K_{\lambda}(h.z_0,k.z_0)\\
   & =  e^{-2\pi i\lambda(x-y)}(1-<h.z_0 , k.z_0 >)^{-\lambda}\\
   & =  e^{-2\pi i\lambda(x-y)} \big{(} 1-\dfrac{1}{2n}\sum_{i=1}^n t_i (s_i)^{-1} \big{)} ^{-\lambda}
\end{align*}
Furthermore, if $f\in \mathcal{S}_{\lambda}$, then the operator $R_{\lambda}R_{\lambda}^*$ becomes
\begin{align*}
     R_{\lambda}R_{\lambda}^* f(t,x) & =\int_{\mathbb{R}\times \mathbb{T}^n} f(s,y) e^{-2\pi i (x-y)} \big{(} 1-\dfrac{1}{2n}\sum_{i=1}^n t_i (s_i)^{-1} \big{)} ^{-\lambda} ds\\
     & = e^{-2\pi i (x-y)} \big{(} \widetilde{f}* \phi_{E(n)} \big{)}(t) 
\end{align*}
where $\widetilde{f}(t)=e^{2\pi i\lambda x}f((t,x).z_0)$ for every $(t,x)\in\widetilde{E(n)}$, and the map $\phi_{E(n)}$ is given by 
\[  \phi_{E(n)}(t)= \big{(} 1-\dfrac{1}{2n}\sum_{i=1}^n t_i  \big{)} ^{-\lambda}\]
where $t\in \mathbb{T}^n$.

  \end{subsection}


\begin{subsection} {Quasi-parabolic}
    The quasi-parabolic subgroup is isomorphic to $\mathbb{T}^{n-1}\times \mathbb{R}$, and it acts on $D_n$, the unbounded realization of the unit ball, by 
    \[ (t,y).(z^{'},z_n)=(tz^{'},z_n+y) \]
    where $t\in \mathbb{T}^{n-1}$, $y\in \mathbb{R}$ and $(z^{'},z_n)\in D_n$.

The Quasi-parabolic can be written as a subgroup of $SU(n,1)$ as

 \[P(n)= \begin{Bmatrix}
  p_{t,y,a}=\begin{pmatrix}
  \begin{smallmatrix}
  at_1 &\\
  & at_2\\
  & & \ddots\\
  &&& at_{n-1}\\
  &&&& a(1+i\dfrac{y}{2}) && a(i \dfrac{y}{2})\\
  &&&& a(-i\dfrac{y}{2}) && a(1-i\dfrac{y}{2})
  \end{smallmatrix}
  
  \end{pmatrix}
  
  \Bigg|\begin{smallmatrix}
   a, t_1,...,t_{n-1} \,\in \, \mathbb{T}\, ,\, y\,\in \,\mathbb{R}\\
  a^{n+1}t_1...t_{n-1}=1
  \end{smallmatrix}
  \end{Bmatrix}
  \]
The group $P(n)$ acts on the unit ball $\mathbb{B}^n$ by 
\[ p_{t,y,a}.(z^{'},z_n)=\Big{(} \dfrac{2}{-iyz_n+2-iy}tz^{'}, \dfrac{(2+iy)z_n+iy}{-iyz_n+2-iy} \Big{)} \]

Let $\widetilde{P(n)}$ be the subgroup of $\widetilde{SU(n,1)}$ that covers $P(n)$. This group can be identified with 
\[\widetilde{P(n)}= \lbrace (t,y,x) \mid  t\in \mathbb{T}^{n-1}, x,y\in \mathbb{R} \quad and \quad e^{2\pi i (n+1)x}t_1...t_{n-1}=1 \rbrace\]
with product 
\[ (t,y,x).(r,w,m)=(tr,y+w,x+m) \]
and the projection map is given by 
\begin{align*}
\widetilde{P(n)} & \rightarrow P(n)\\
(t,y,x) & \mapsto p_{t,y,e^{2\pi i x}}  
\end{align*}
Let $z_0= \Big{(} \dfrac{1}{\sqrt{2(n-1)}},...,\dfrac{1}{\sqrt{2(n-1)}},0 \Big{)}$ be in $\mathbb{B}^n$. For each $q=(t,y,x)\in\widetilde{P(n)}$ we have
\[ D_{\lambda}(q)=2^{\lambda}e^{-2\pi i \lambda x}(2-iy)^{-\lambda}\]
 The operator $N=y\dfrac{\partial}{\partial\,y}$,
and direct calculations shows that 
\[ A_{\lambda}B_{\lambda}\big{(}D_{\lambda}(q)\big{)}= 2^{\lambda+2m}e^{-2\pi i \lambda x}(2-iy)^{-(\lambda+2m)} \]
Also, we have
\begin{align*}
    |A_{\lambda}B_{\lambda}\big{(}D_{\lambda}(q)\big{)}|^2 & = 2^{2(\lambda+2m)}|2-iy|^{-2(\lambda+2m)}\\
    &= \Big{(} 1+ \dfrac{|y|^2}{4} \Big{)}^{-(\lambda+2m)}
\end{align*}
Therefore, $D_{\lambda}\in \mathcal{S}_{\lambda}$ for all $\lambda>\dfrac{1}{2}-2m$.

Let $h=(t,y,x)$ and $k=(t^{'},y^{'},x^{'})$. The operator $R_{\lambda}(h,k)$ can be written as
\begin{align*}
    R_{\lambda}(h,k) & =  D_{\lambda}(h) Q\,\overline{A_{\lambda}B_{\lambda} \big ( D_{\lambda}(k) K_{\lambda}(k.z_0,h.z_0)\big )}\\ & = e^{-2\pi i (x-x^{'})} \, Q\overline{A_{\lambda}B_{\lambda} \Big{(} 1-\dfrac{i}{2}(y^{'}-y)-\dfrac{1}{2(n-1)}\sum_{i=1}^{n-1} t_i^{'} t_i^{-1} \Big{)}^{-\lambda}}\\
    &= e^{-2\pi i (x-x^{'})} Q\, A_{\lambda}B_{\lambda} \Big{(} \big{(}1-\dfrac{i}{2}(y-y^{'})-\dfrac{1}{2(n-1)}\sum_{i=1}^{n-1} t_i (t_i^{'})^{-1} \big{)} ^{-\lambda}\Big{)}
\end{align*}
where $Q$, $A_{\lambda}$ and $B_{\lambda}$ acts on $y^{'}$. Moreover, the operator $Q$ has the form 
\[ Q=c_1+c_2 y\dfrac{\partial}{\partial y}+c_3 y^2 \dfrac{ \partial^2}{\partial y^2}+...+c_{2m} y^{2m}\dfrac{\partial^{2m}}{\partial y^{2m}} \]
where $c_i's$ are constants .

If $f\in \mathcal{S}_{\lambda}$, then the operator $R_{\lambda}R_{\lambda}^*$ becomes 
\begin{align*}
    R_{\lambda}R_{\lambda}^* f(t,y,x)&=\int\limits_{\mathbb{R}\times \mathbb{T}^{n-1}} f(t^{'}, y^{'}, z^{'}) e^{-2\pi i (x-x^{'})} \,\, Q\,A_{\lambda}B_{\lambda} \Big{(} \big{(} 1-\dfrac{i}{2}(y-y^{'}) \\& \quad\quad\quad -\dfrac{1}{2(n-1)}\sum_{i=1}^{n-1} t_i (t_i^{'})^{-1} \big{)}^{-\lambda}\Big{)}dt^{'}dy^{'} \\
    & = e^{-2\pi i \lambda x} (\widetilde{f}*\phi_{P(n)})(t,y)
\end{align*}
where $\widetilde{f}(t,y)=e^{2\pi i\lambda x}f(t,y,x)$ and the function $\phi_{P(n)}$ is given by 
\[\phi_{P(n)}(t,y)= QA_{\lambda}B_{\lambda} \Big{(} \big{(} 1-\dfrac{i}{2}y-\dfrac{1}{2(n-1)}\sum_{i=1}^{n-1} t_i \big{)}^{-\lambda}\Big{)}\]

The function $\phi_{P(n)}$ is in $L^1(P(n))$ for all $\lambda>0$. To prove this, first notice that 
\[ A_\lambda B_\lambda \big{(} 1-\dfrac{i}{2}y-\dfrac{1}{2(n-1)}\sum_{i=1}^{n-1} t_i \big{)}^{-\lambda}= \Big{(}1-\dfrac{1}{2(n-1)}\sum_{i=1}^{n-1} t_i \Big{)}^{2m} \Big{(}1-\dfrac{i}{2}y-\dfrac{1}{2(n-1)}\sum_{i=1}^{n-1} t_i  \Big{)}^{-(\lambda+2m)} \]
since $\big{|}\dfrac{1}{2(n-1)}\sum_{i=1}^{n-1} t_i \big{|} \leq 1$, then by following an argument similar to \cite [p.~212]{Dawson2018} we get 
\begin{equation}
\label{equality1}
    \big{|} \Big{(}1-\dfrac{i}{2}y-\dfrac{1}{2(n-1)}\sum_{i=1}^{n-1} t_i  \Big{)}^{-(\lambda+2m)} \big{|}\leq \big{|} 1+(|y|-1)^2 \big{|}^ {-(\lambda+2m)/2}
\end{equation}
Since the operator $Q$ is a polynomial differential operator, then it is enough to consider $y^r\dfrac{\partial^r}{\partial y^r}$.
Equation \ref{equality1} implies, $y^r \Big{(}1-\dfrac{i}{2}y-\dfrac{1}{2(n-1)}\sum_{i=1}^{n-1} t_i  \Big{)}^{-(\lambda+2m+r)} $ is in $L^1(P(n))$ for all $0\leq r \leq 2m$.
Therefore, $\phi_{P(n)}$ is in $L^1(P(n))$ for all $\lambda>1-2m$.

\end{subsection}

\begin{subsection}{Quasi-hyperbolic}
    The quasi-hyperbolic subgroup is isomorphic to $\mathbb{T}^{n-1}\times \mathbb{R}^+$ and it acts on the unbounded domain $D_n$ by
    \[ (t,r).(z^{'},z_n)=(rtz^{'},r^2z_n) \]

As a subgroup of $SU(n,1)$ it can be written as \cite{Dawson2018}

\[H(n)= \begin{Bmatrix}
  
  h_{t,s,a}=\begin{pmatrix}
  
  \begin{smallmatrix}
  at_1 &\\
  & at_2\\
  & & \ddots\\
  &&& at_{n-1}\\
  &&&& a\,\cosh\,s && a\, \sinh\,s\\
  &&&& a\,\sinh\,s && a\,\cosh\,s
  \end{smallmatrix}
  
  \end{pmatrix}
  
  \Bigg|\begin{smallmatrix}
   a, t_1,...,t_{n-1} \,\in \, \mathbb{T}\, ,\, s\,\in \,\mathbb{R}\\
  a^{n+1}t_1...t_{n-1}=1
  \end{smallmatrix}
  \end{Bmatrix}
\]
and it acts on the unit ball $\mathbb{B}^n$ by
\[ h_{t,s,a}.(z^{'},z_n)=\Big{(}\dfrac{tz^{'}}{z_n\,\sinh\,s+\cosh\,s},\dfrac{z_n\,\cosh\,s+\sinh\,s}{z_n\,\sinh\,s+\cosh\,s} \Big{)} \]
The group $H(1)$ is simply connected. So we have $ H(1)\simeq \widetilde{H(1)}\simeq\mathbb{R}$.
When $n>1$, then $\widetilde{H(n)}$ can be identified with
\[ \widetilde{H(n)}= \lbrace (t,s,x)\mid  t\in \mathbb{T}^{n-1},\, s,x \in \mathbb{R}\,\mid  e^{2{\pi}i(n+1)x}t_1t_2...t_{n-1}=1\rbrace\] with 
\[ (t,s,x).(t^{'},s^{'},x^{'})=(tt^{'},s+s^{'},x+x^{'}) \]

Let $z_0= \big{(} \dfrac{1}{\sqrt{2(n-1)}},...,\dfrac{1}{\sqrt{2(n-1)}},0 \big{)}$, then we have 
\[ h_{t,s,a}.z_0=\big{(} \dfrac{t}{\cosh\,s}z_0^{'} ,\tanh\,s \big{)}\]
So, we have $H(n).z_0\simeq H(n)/{H(n)_{z_0}}\simeq \mathbb{T}^{n-1}\times \mathbb{R}$.

Let $q=(t,s,x)\in\widetilde{H(n)}$, then we have $D_{\lambda}(q)=e^{-2\pi i \lambda x}\big{(} \cosh\,s\big{)}^{-\lambda}$. The operator $N$ is defined by \[ N=\dfrac{\sinh\,s}{\cosh\,s}\dfrac{\partial}{\partial\,s}\]
Direct calculations show that 
\begin{equation}
\label{differential of cosh}
    \big{(}I+\dfrac{N}{r}\big{)} (\cosh\,s)^{-l}=(r-l)r^{-1} (\cosh\,s)^{-l}+lr^{-1} (\cosh\,s)^{-(l+2)}
\end{equation}
and for all $l>0$, we have \[\int_{\mathbb{R}} (\cosh\,s)^{-l}\,ds\,<\infty\]
This with equation \ref{differential of cosh}  implies \[\int_{\mathbb{R}} A_{\lambda}B_{\lambda}\big{(} (\cosh\,s)^{-\lambda} \big{)}^2\, ds\, <\infty \]
Therefore, $D_{\lambda}$ is in $\mathcal{S}_{\lambda}$ for all $\lambda>0$.

Let $h=(t,s,x)$ and $k=(t^{'},s^{'},x^{'})$. The operator $R_{\lambda}(h,k)$ can be written as
\begin{align*}
    R_{\lambda}(h,k) & =  D_{\lambda}(h) Q\,\overline{A_{\lambda}B_{\lambda} \big ( D_{\lambda}(k) K_{\lambda}(k.z_0,h.z_0)\big )}\\ & = e^{-2\pi i (x-x^{'})} \, Q\overline{A_{\lambda}B_{\lambda} \Big{(} \cosh(s^{'}-s)-\dfrac{1}{2(n-1)}\sum_{i=1}^{n-1} t_i^{'} t_i^{-1} \Big{)}^{-\lambda}}\\
    &= e^{-2\pi i (x-x^{'})} Q\, A_{\lambda}B_{\lambda} \Big{(} \big{(}\cosh(s-s^{'})-\dfrac{1}{2(n-1)}\sum_{i=1}^{n-1} t_i (t_i^{'})^{-1} \big{)} ^{-\lambda}\Big{)}
\end{align*}
where $Q$, $A_{\lambda}$ and $B_{\lambda}$ acts on $s^{'}$. Moreover, the operator $Q$ has the form 
\[ Q=a_1(s)+a_2(s) \tanh(s)\dfrac{\partial}{\partial s}+a_3 {\tanh^2(s)} \dfrac{ \partial^2}{\partial s^2}+...+a_{2m}(s) {\tanh^{2m}(s)}\dfrac{\partial^{2m}}{\partial s^{2m}} \]
and $a_i(s)$ are functions 
\[ a_i(s)=\big{(}\sum_j c_j \dfrac{\partial^j}{\partial\,s^j}\big{)} \tanh\,s =\sum_{j}c_j \tanh^{(j)}(s) \] where $c_i's$ are constants.

If $f\in \mathcal{S}_{\lambda}$, then the operator $R_{\lambda}R_{\lambda}^*$ becomes 
\begin{align*}
    R_{\lambda}R_{\lambda}^* f(t,y,x)&=\int\limits_{\mathbb{R}\times \mathbb{T}^{n-1}} f(t^{'}, y^{'}, z^{'}) e^{-2\pi i (x-x^{'})} \, Q\, A_{\lambda}B_{\lambda} \Big{(} \big{(}\cosh(s-s^{'})\\ & \quad\quad\quad\quad\quad\quad\quad\quad\quad\quad\quad\quad\quad\quad -\dfrac{1}{2(n-1)}\sum_{i=1}^{n-1} t_i (t_i^{'})^{-1} \big{)} ^{-\lambda}\Big{)} dt^{'}ds^{'}\\
    & = e^{-2\pi i \lambda x} (\widetilde{f}*\phi_{H(n)})(t,s)
\end{align*}
where $\widetilde{f}(t,s)=e^{2\pi i\lambda x}f(t,s,x)$ and the function $\phi_{P(n)}$ is given by 
\[\phi_{H(n)}(t,s)= QA_{\lambda}B_{\lambda} \Big{(} \big{(} \cosh\,s-\dfrac{1}{2(n-1)}\sum_{i=1}^{n-1} t_i \big{)}^{-\lambda}\Big{)}\]

\begin{remark}{\rm
    Direct calculations shows that 
    \[
        A_{\lambda}B_{\lambda}\big{(} (\cosh s)^{-\lambda}\big{)}=\sum_{j=1}^{2m} c_j (\cosh s)^{-(\lambda+2j)}
    \]
    where $c_j's$ are constants. Also, we have \[ A_{\lambda}B_{\lambda}\Big{(} \big{(}\cosh s\, -\dfrac{1}{2(n-1)}\sum_{i=1}^{n-1} t_i\big{)}^{-\lambda}\Big{)}=\sum_{j=1}^{2m} \gamma_j(s) \big{(}\cosh s-\dfrac{1}{2(n-1)}\sum_{i=1}^{n-1} t_i\big{)}^{-(\lambda+2j)} \]
    where $\gamma_j(s)$ are functions that are constant multiples of $\sech\,s$ and $\tanh s$. Now consider 
    \begin{align*}
        \dfrac{\partial}{\partial s}\big{(} \gamma_j(s) \big{(}\cosh s-\dfrac{1}{2(n-1)}\sum_{i=1}^{n-1} & t_i\big{)}^{-(\lambda+2j)} \big{)}= \gamma_j^{'}(s)\big{(}\cosh s-\dfrac{1}{2(n-1)}\sum_{i=1}^{n-1} t_i\big{)}^{-(\lambda+2j)} \\ & + \gamma_j(s)(-\lambda-2j)\sinh s \,\, \big{(}\cosh s-\dfrac{1}{2(n-1)}\sum_{i=1}^{n-1} t_i\big{)}^{-(\lambda+2j+1)}
    \end{align*}

    It can be readily seen that
    \[ \dfrac{\partial^k}{\partial s^k} \Big{(}\gamma_j(s) \big{(}\cosh s-\dfrac{1}{2(n-1)}\sum_{i=1}^{n-1} t_i\big{)}^{-(\lambda+2j)} \Big{)} =\sum_{r=0}^k a_r(s)  \big{(}\cosh s-\dfrac{1}{2(n-1)}\sum_{i=1}^{n-1} t_i\big{)}^{-(\lambda+2j+r)}  \]
    where $a_r(s)$ is a linear combination of $\gamma_j(s)^{(i_r)}$ and powers of ${\sinh}^{q_r}$ and ${\cosh}^{w_r}$, where the powers $q_r,\, w_r$ are strictly less than $\lambda+2j+r$. Since \[\small{\dfrac{1}{2(n-1)}}\displaystyle\sum_{i=1}^{n-1} t_i\leq \dfrac{1}{2}\] it follows that  
    \[ \big{|} \cosh s\, -\dfrac{1}{2(n-1)}\sum_{i=1}^{n-1} t_i\big{|}^{-\lambda}\leq \big{|} \cosh s-\dfrac{1}{2}\big{|}^{-\lambda}\]
    Moreover, $|\tanh\,s|$ and $|\sech\,s|$ are bounded above by 1, and \[ \int_{\mathbb{R}}(\sinh s)^{q} (\cosh s-\dfrac{1}{2})^{-v}\,ds < \infty \]
     \[ \int_{\mathbb{R}}(\cosh s)^{w} (\cosh s-\dfrac{1}{2})^{-v}\,ds < \infty \]
    for all $q,w < v$ where $v>1$.
    Therefore, we have $\phi_{H(n)}\in L^1(H(n).z_0)$ for all $\lambda>0$.
 }   
\end{remark}

\end{subsection}

\section{ $T_{\varphi}^{(\lambda)}$ with Symbols Invariant under Maximal Abelian Subgroups}

   In this section we will show that the analytic continuation of Toeplitz operators can be written as a convolution operator on sections of line bundle. It was shown in \cite[Theorem 5.1]{Dawson2018} that for $\lambda>n$, the operator 
  \[ R_{\lambda}T_{\varphi}^{(\lambda)}R_{\lambda}^* :  \mathcal{S}_{\lambda}\rightarrow \mathcal{S}_{\lambda}  \]
  can be written as
  \[ R_{\lambda}T_{\varphi}^{(\lambda)}R_{\lambda}^* f=f*\nu_{\varphi} \]
  
  We will show that we can extend it to cover the analytic continuation case. First, we need the following Lemma, which shows how the representation $\pi_{\lambda}$ acts on the kernel of $A_{\lambda}^2(\mathbb{B}^n)$.
  
  \begin{lemma}
  Let $g\in G$, and $z\in \mathbb{B}^n$, we have \[ \pi_{\lambda}(g) K_{z}=\overline{j_{\lambda}(g,z)}K_{g.z} \]
  
  \end{lemma}

\begin{theorem}
    \label{Theorem as a convolution operator}
   Let $\lambda_0>0$ and $H$ be a maximal abelian subgroup of $SU(n,1)$ such that $H.z_0$ is restriction injective. Assume $\varphi\in BC^n(\mathbb{B}^n)^H$ is $H-$invariant, then 
   \[R_{\lambda_0}T_{\varphi}^{(\lambda_0)}R_{\lambda_0}^*f=A_{\lambda_0}B_{\lambda_0}(f)*\nu_{\varphi}^{\lambda_0}  \]
   where $\nu_{\varphi}^{\lambda_0}:H/H_{z_0}\rightarrow \mathbb{C}$ is given by
   
   \[\nu_{\varphi}^{\lambda_0}(h.z_0):=\An_{\lambda_0}\Big{(} \overline{ A_{\lambda_0}B_{\lambda_0} \big{(} <   \overline{j_{\lambda}(h,z_0)} K_{h.z_0,\,\lambda} , \,\,\, \varphi  K_{z_0,\lambda}  >_{\lambda} \big{)}} \Big{)} \]
and $\An_{\lambda_0}$ denotes taking the analytic continuation in $\lambda$ evaluated at $\lambda_0$
\end{theorem}
\begin{proof} Let $h\in \widetilde{H}$, we have
\begin{align*}
    R_{\lambda_0}T_{\varphi}^{(\lambda_0)}R_{\lambda_0}^*f(&h) =D_{\lambda_0}(h)T_{\varphi}^{(\lambda_0)}R_{\lambda_0}^*f(h.z_0)\\
    &= D_{\lambda_0}(h) < T_{\varphi}^{(\lambda_0)}R_{\lambda_0}^*f, K_{h.z_0,\,\lambda_0}>_{\lambda_0}\\
    &= D_{\lambda_0}(h) \An_{\lambda_0}\Big{(} < T_{\varphi}^{(\lambda)}R_{\lambda_0}^*f, K_{h.z_0,\,\lambda_0}>_{A_{\lambda}^2}\Big{)}\\
   &=  D_{\lambda_0}(h) \An_{\lambda_0}\Big{(} < \varphi R_{\lambda_0}^*f, K_{h.z_0,\,\lambda_0}>_{A_{\lambda}^2}\Big{)}\\
    & = D_{\lambda_0}(h) \An_{\lambda_0}\Big{(} \int_{\mathbb{B}^n} \varphi(z) R_{\lambda_0}^*f(z) \overline{K_{h.z_0,\,\lambda_0}(z)} \, d\mu_{\lambda}(z) \Big{)}\\
    & = D_{\lambda_0}(h) \An_{\lambda_0}\Big{(} \int_{\mathbb{B}^n} \varphi(z) \int_{H/H_{z_0}} A_{\lambda_0}B_{\lambda_0}\big{(}f(k)\big{)} \overline{A_{\lambda_0}B_{\lambda_0}\big{(}D_{\lambda_0}(k)K_{z,\lambda_0}(k.z_0)} \big{)}\, dk \\ & \quad\quad\quad\quad \quad\quad\quad\quad \quad\quad\quad\quad  \overline{K_{h.z_0,\,\lambda_0}(z)} \, d\mu_{\lambda}(z) \Big{)}
\end{align*}
So, the operator $R_{\lambda_0}T_{\varphi}^{(\lambda_0)}R_{\lambda_0}^*f(h)$ can be written as 
\[\An_{\lambda_0}\Big{(} \int\limits_{\mathbb{B}^n} \varphi(z) \Big{[} \int\limits_{H/H_{z_0}} A_{\lambda_0}B_{\lambda_0}\big{(}f(k)\big{)} \overline{\strut A_{\lambda_0}B_{\lambda_0}\big{(}D_{\lambda_0}(k)K_{z,\lambda_0}(k.z_0) \overline{D_{\lambda_0}(h)} K_{h.z_0,\,\lambda_0}(z) } \big{)} dk \Big{]}  \, dz \Big{)}\]
and by interchanging the order of integration we get
\[\An_{\lambda_0} \int\limits_{H/H_{z_0}} A_{\lambda_0}B_{\lambda_0}\big{(}f(k)\big{)} \overline{\strut A_{\lambda_0}B_{\lambda_0}\big{(} <  \overline{ \, j_{\lambda}(h,z_0)} K_{h.z_0,\,\lambda}, \,\,\, \varphi \, \overline{j_{\lambda}(k,z_0)} K_{k.z_0,\lambda}>_{\lambda}  \big{)} } d\mu(k)\]

For $k\in \widetilde{H}$, define $W_k:= <  \overline{\, j_{\lambda}(h,z_0)} K_{h.z_0,\,\lambda}, \,\, \varphi \overline{j_{\lambda}(k,z_0)} K_{k.z_0,\lambda}>_{\lambda}$. Then $W_k$ can be written as

\begin{align*}
 W_k  &= <  \overline{ \, j_{\lambda}(h,z_0)}\pi(k^{-1}) K_{h.z_0,\,\lambda}, \,\,\, \overline{j_{\lambda}(k,z_0)}\pi(k^{-1})\big{(}\varphi K_{k.z_0,\lambda} \big{)} >_{\lambda}\\
   & = <  \overline{ \, j_{\lambda}(h,z_0)} \overline{j_{\lambda}(k^{-1},h.z_0)} K_{k^{-1}h.z_0,\,\lambda} \, , \,\,\, \varphi \, \overline{j_{\lambda}(k,z_0)}\overline{j_{\lambda}(k^{-1},k.z_0)} K_{z_0,\lambda}  >_{\lambda}\\
   & = <   \overline{j_{\lambda}(k^{-1}h,z_0)} K_{k^{-1}h.z_0,\,\lambda} , \,\,\, \varphi  K_{z_0,\lambda}  >_{\lambda}
\end{align*} 
Therefore,
\[ R_{\lambda_0}T_{\varphi}^{(\lambda_0)}R_{\lambda_0}^* f(h)=\int\limits_{H/H_{z_0}} A_{\lambda_0}B_{\lambda_0}(f(k)) \,\An_{\lambda_0} \Big{(} \overline{\strut A_{\lambda_0}B_{\lambda_0}\big{(}  <    \overline{j_{\lambda}(k^{-1}h,z_0)} K_{k^{-1}h.z_0,\,\lambda} ,  \varphi  K_{z_0,\lambda}  >_{\lambda}  }\big{)}\, \Big{)}\qedhere\]
\end{proof} 

Integration by parts implies the following corollary.
  \begin{corollary} \label{corollary convolution operator}
      The operator $R_{\lambda_0}T_{\varphi}^{(\lambda_0)}R_{\lambda_0}^*$ can be written as
      \[ R_{\lambda_0}T_{\varphi}^{(\lambda_0)}R_{\lambda_0}^*f(h)= f* Q\,\nu_{\varphi}^{\lambda_0}\]
      where $Q$ is a polynomial differential operator.
      \end{corollary}

 For $f\in \mathcal{S}_{\lambda}$, we can write
 \begin{align*}
      U_{\lambda}^* T_{\varphi}^{(\lambda)}U_{\lambda}f & = (\sqrt{R_{\lambda}R_{\lambda}^*})^{-1} \, R_{\lambda}T_{\varphi}^{(\lambda)}R_{\lambda}^* (\sqrt{R_{\lambda}R_{\lambda}^*})^{-1}\\
  & = (R_{\lambda}R_{\lambda}^*)^{-1} \, R_{\lambda}T_{\varphi}^{(\lambda)}R_{\lambda}^*
 \end{align*}
  This with Corollary \ref{corollary convolution operator} implies the following proposition, which gives the spectral representation of the Toeplitz operators.
  
 \begin{theorem}
 Let $\lambda>0$, $H$ be a maximal abelian subgroup of $SU(n,1)$, and let 
 \[A=\lbrace \psi \in \widehat{H/H_{z_0}}\mid \widehat{\phi_H^{\lambda}}(\psi)\neq0 \rbrace.\]
 If $\varphi\in BC^n(\mathbb{B}^n)^H$ is $H-$invariant. Then 
 \[ \mathscr{F}U_{\lambda}^{-1}T_{\varphi}^{(\lambda)}U_{\lambda}\mathscr{F}^{-1}\omega(\psi)=\dfrac{\widehat{Q\nu_{\varphi}^{\lambda}}(\psi)}{\widehat{\phi_H^{\lambda}}(\psi)}\omega(\psi)\]
 for all $\psi \in A$, and $\omega\in L^2(\widehat{H/H_{z_0}})$ such that $\supp \omega\subseteq \supp  \widehat{\phi_H^{\lambda}} $. Here $\mathscr{F}f(\psi)= \mathscr{F}_{H/H_{z_0}}\widetilde{f}(\psi) $.
 \end{theorem}
 
\nocite{*}
\begin{sloppypar}
   \printbibliography 
\end{sloppypar}

\end{document}